\documentclass[11pt, twoside, a4paper]{article}
\usepackage[margin = 1in]{geometry}
\usepackage{comment}
\usepackage{indentfirst}
\usepackage[pdftex]{graphicx}
\usepackage{amsmath}
\usepackage{amssymb}
\usepackage{framed}
\usepackage{aliascnt}
\usepackage{hyperref}
\usepackage{multirow, multicol}
\usepackage{subfigure, subfloat, graphicx} 
\usepackage{verbatim, rotating, paralist}
\usepackage{caption}
\usepackage{pstricks, sgamevar, egameps}
\usepackage[framed, hyperref, thmmarks, amsmath]{ntheorem}
\theoremseparator{.}

\newcommand{\dmwa}{\mbox{discrete replicator dynamics}} 
\newcommand{\edmwa}{\mbox{\em discrete replicator dynamics}} 
\newcommand{\Real}{\mbox{${\mathbb R}$}}
\newcommand{\RJ}{\mbox{${\mathbb J}$}}

\newcommand{\xx}{\mbox{\boldmath $x$}}
\newcommand{\yy}{\mbox{\boldmath $y$}}
\newcommand{\vv}{\mbox{\boldmath $v$}}
\newcommand{\rr}{\textbf{r}}
\newcommand{\qq}{\mbox{\boldmath $q$}}
\newcommand{\kk}{\mbox{\boldmath $k$}}
\newcommand{\zz}{\mbox{\boldmath $z$}}
\newcommand{\eps}{\mbox{$\epsilon$}}
\newcommand{\eeps}{\mbox{\boldmath $\epsilon$}}
\newcommand{\pp}{\mbox{\boldmath $p$}}
\newcommand{\D}{\mbox{$\Delta$}}
\newcommand{\ones}{\mbox{${\bf 1}$}}

\newcommand{\defeq}{\stackrel{\textup{def}}{=}}
\DeclareMathOperator*{\argmax}{arg\,max}

\newtheorem{thm}{Theorem}
\newaliascnt{lem}{thm}
\newtheorem{lem}[lem]{Lemma}
\newtheorem{definition}[thm]{Lemma}
\newtheorem{claim}[thm]{Claim}
\newtheorem{remark}[thm]{Remark}
\aliascntresetthe{lem}

\newaliascnt{mydef}{thm}
\newtheorem{mydef}[mydef]{Definition}
\aliascntresetthe{mydef}

\newaliascnt{prop}{thm}

\aliascntresetthe{prop}

\newaliascnt{clo}{thm}

\aliascntresetthe{clo}

\theoremnumbering{Alph}
\newframedtheorem{algo}{Algorithm}

\theoremnumbering{arabic}

\theorembodyfont{\upshape}
\newframedtheorem{algoa}{Algorithm}

\theoremseparator{}
\theoremsymbol{\rule{1ex}{1ex}}
\theoremstyle{nonumberplain}
\newtheorem{proof}{Proof}

\newcommand{\etal}{{\em et al.~}}

\usepackage{quote}

\usepackage{titling}
\newcommand{\subtitle}[1]{%
  \posttitle{%
    \par\end{center}
    \begin{center}\large#1\end{center}
    \vskip0.5em}%
}

\begin{document}

\title{Natural Selection as an Inhibitor of Genetic Diversity} 
\subtitle{Multiplicative Weights Updates Algorithm and a Conjecture of Haploid Genetics}

\author{
Ruta Mehta\\Georgia Institute of Technology\\ rmehta@cc.gatech.edu
\and Ioannis Panageas\\Georgia Institute of Technology\\ioannis@gatech.edu
\and Georgios Piliouras\\California Institute of Technology\\gpiliour@caltech.edu}

\date{}

\maketitle

\thispagestyle{empty}

\begin{abstract}
In a recent series of papers \cite{Arxiv:DBLP:journals/corr/abs-1208-3160,ITCS:DBLP:dblp_conf/innovations/ChastainLPV13,PNAS2:Chastain16062014}  a surprisingly strong connection was discovered between standard models of evolution in mathematical biology and Multiplicative Weights Updates Algorithm, a ubiquitous model of online learning and optimization. These papers  establish that mathematical models of biological evolution are tantamount to applying \dmwa~\cite{akin,Hofbauer98}, a close variant of MWUA, on coordination games. This connection allows for introducing insights from the study of game theoretic dynamics into the field of mathematical biology.

Using these results as a stepping stone, we show that  mathematical models of haploid evolution imply the
extinction of genetic diversity in the long term limit, a widely believed conjecture in genetics
\cite{barton2014diverse}.  In game theoretic terms we show that in the case of coordination games, under minimal genericity
assumptions, \dmwa~converge to pure Nash equilibria for all but a zero measure of initial conditions. This result holds despite the fact
 that mixed Nash equilibria can be exponentially (or even uncountably) many, completely dominating in number the set of pure
 Nash equilibria.
Thus, in haploid organisms the long term preservation of
 genetic diversity needs to be safeguarded by other evolutionary mechanisms such as mutations and
  speciation.

\end{abstract}

\newpage
\clearpage
\setcounter{page}{1}

\section{Introduction}

Decoding the mechanisms 
 of biological evolution has been one of the most inspiring contests for the human mind. The modern theory of population genetics has been derived by combining the Darwinian concept of natural selection and Mendelian genetics.  Detailed experimental studies of a species of fruit fly, Drosophila, allowed for
 a unified understanding of evolution that encompasses both
the Darwinian view of 
continuous evolutionary improvements and the discrete nature of Mendelian genetics.
The key insight is that evolution relies on the progressive selection of organisms with advantageous
mutations.  This understanding has lead to precise mathematical formulations of such evolutionary mechanisms,
dating back to the work of Fisher, Haldane, and Wright \cite{HistoryofEvolutionIdea} in the beginning of the twentieth century.

The existence of dynamical models of genotypic evolution, however, does not  offer  by itself   clear, concise insights
about the future states of the phenotypic landscape.\footnote{See Section \ref{sec:prelEvol} for (non-technical) definition of biological terms.} Which allele combinations, and as a result, which attributes will take over?
  \textit{Prediction of the evolution
 of the phenotypic landscape is a key, alas not well understood, question in the study of biological systems \cite{quanta14}.}

 Despite the advent of detailed mathematical models, still at the forefront of our understanding lie experimental studies and simulations. Of course, this is to some extent inevitable since the involved dynamical systems are nonlinear and  hence a complete theoretical understanding of all related questions seems intractable \cite{Scriven28081959,EVO:EVO12354}.
Nevertheless, some rather useful qualitative statements have been established.

Nagylaki \cite{Nagylaki1} showed that, when mutations do not affect reproduction success by a lot\footnote{This is referred to as the weak selection regime and it corresponds to a well supported principle known as Kimura's neutral theory.}, the system state converges quickly to the
 so-called Wright manifold, where the distribution of genotypes is a product distribution of the allele frequencies in the population.
In this case, in order to keep track of the distribution of genotypes in the population it suffices to record the distribution of the different alleles for each gene. The overall distribution of genotypes can be recovered by simply taking products of the allele frequencies. Nagylaki \etal  \cite{Nagylaki2}  have also shown that under hyperbolicity assumptions (\textit{e.g.}, isolated equilibria) such systems converge.

Recently, Chastain \etal has built on Nagylaki's  work by establishing an informative connection between these mathematical models of
population genetics and the multiplicative update algorithm (MWUA).  MWUA is a ubiquitous online learning dynamic
\cite{Arora05themultiplicative}, which is known to enjoy numerous connections to biologically relevant mathematical models.
Specifically, its continuous time limit is equivalent to the replicator dynamics (in its standard continuous
form)~\cite{Kleinberg09multiplicativeupdates} and its equivalent up to a smooth change of variables to the
Lotka-Volterra equations~\cite{Hofbauer98}. In
\cite{Arxiv:DBLP:journals/corr/abs-1208-3160,ITCS:DBLP:dblp_conf/innovations/ChastainLPV13} another strong
such connection was established. Specifically, under the assumption of weak selection 
 standard models of population
genetics are shown to be closely related to applying discrete replicator dynamics \footnote{This
 MWUA variant, which Chastain \etal refer to as discrete MWUA, is already a well established dynamic in the literature of mathematical biology and evolutionary game theory~\cite{akin,Hofbauer98} under the name discrete (time, version of) replicator dynamics and to avoid confusion we will refer to it by its standard name.} on a coordination game.

The resulting coordination game is as follows: Each gene is an agent and its available strategies are its alleles. Any combination of
strategies/alleles (one for each gene/agent) gives rise to a specific genotype/individual. The common utility of each gene/agent at
that genotype/outcome is equal to the fitness of that phenotype\footnote{In the weak selection regime this is a number in [1-s,1+s] for
some small $s>0$.}. 
If we interpret the frequency of the allele in the population as mixed (randomized) strategies in this  game then the population
genetics model reduces to each agent/gene updating their distribution according to \dmwa.

This connection allows for the translation of results between the areas of genetics and game theory. We begin with a brief overview of some key insights that have emerged  thus far.

In \dmwa~ the rate of increase of the probability of a given strategy is directly proportional to its current expected utility. In population genetic terms, this expected utility reflects the average fitness of a specific allele when matched with the current mixture of alleles of the other genes. Livnat \etal \cite{PNAS1:Livnat16122008} coined  the term mixability  to refer to this beneficial attribute. In other words, an allele with high mixability achieves high fitness when paired against the current allele distribution. Naturally, this trait is not a standalone characteristic of an allele but depends on the current state of the system, \textit{i.e.}, the other allele frequencies. An allele that enjoys a high mixability in one distribution of alleles, might exhibit a low mixability in another. So,  although mixability offers a palpable interpretation of how evolutionary models behaves in a single time step, it does not offer insights about the long term behavior.

Game theory, however, can provide us with clues about the long term behavior as well. Specifically,
\dmwa~converges to sets of fixed points in variants of coordination games \cite{akin,Soda14}. 
This  allows for a concise characterization of the possible limit points of the population genetics model, since they coincide
with the set of equilibria (fixed-points). \footnote{A mixed strategy profile is an equilibrium of our system if
and only if, for each agent,  all strategies which are played with positive probability have equal expected utility. This
set encompasses the set of Nash equilibria of the underlying game, since they furthermore require that any
strategy that is not played by any agent must result in expected utility that is no greater than his current expected
payoff. This requirement is not present in \dmwa, since it does not explore new strategies.}.  In
\cite{Arxiv:DBLP:journals/corr/abs-1208-3160,ITCS:DBLP:dblp_conf/innovations/ChastainLPV13} it was observed that random two
agent coordination games (in the weak selection payoff regime) exhibit (in expectation) exponentially many such mixed
strategies. The abundance of such mixed Nash equilibria seems like a strong indicator that (i) the long term system
behavior will result in a state of high genetic variance (highly mixed population), (ii) we cannot even efficiently
enumerate the set of all biologically relevant limiting behaviors, let alone predict them.  We show that this
intuition does not reflect accurately the dynamical system behavior.
\medskip

\noindent\textbf{Our game theoretic results.} We show that given a generic two agent coordination games, starting from all but a zero
measure of initial conditions, discrete MWUA converges to pure, strict Nash equilibria. The genericity assumption is minimal
and merely requires that
any row/column of the payoff matrix have distinct entries\footnote{This genericity
assumption, for example, is trivially satisfied with probability one, if the entries of the matrix are iid from a
distribution that is continuous and symmetric around zero, say uniform in $[-1, 1]$ as in
\cite{Arxiv:DBLP:journals/corr/abs-1208-3160,ITCS:DBLP:dblp_conf/innovations/ChastainLPV13}. This class of games contains instances
with uncountably many Nash equilibria, i.e., $A=\left[\begin{array}{cc}1& 4\\ 4 &1\\ 2 &3\end{array}\right]$.}. Our results carry over even
if the game has uncountably many Nash equilibria.

Conceptually, our paper is based on one key idea. \textit{Equilibria (fixed-points) of a dynamical system may be unstable.}
An example of an unstable equilibrium is an ideal coin that lies on its edge. This is a fixed point/equilibrium of the dynamical system (on paper), but it has no predictive value in terms of actual system behavior.
It corresponds to a probability zero event in the sense that for the coin to land on its edge the set of  allowable starting conditions is negligible. Even if we try to place an idea coin on its edge, the inherent instability of the state will amplify even the most minute of disturbances fast and cause it to topple  and land on one of its two stable equilibria, either heads or tails. If we think of this knife edge equilibrium  as encoding a mixed state (symmetric 50\% heads 50\% tails), in the resulting stable states this symmetry has collapsed and we end up with a pure state. This  high level intuition captures the essence  behind our theorem. This mixability/symmetry breaking argument is \textit{universal} in the sense that it holds for all mixed states.

Technically, our paper is based mostly on two prior works. In \cite{Kleinberg09multiplicativeupdates} the generic instability of mixed Nash  was established for other variants of MWUA, including the replicator equation.
Our instability analysis follows along similar lines. Any linearly stable equilibrium is shown to be a weakly stable Nash equilibrium\footnote{This Nash equilibrium refinement was introduced in \cite{Kleinberg09multiplicativeupdates}. A weakly stable Nash is a Nash equilibrium that satisfies the extra property that if you force any single randomizing agent to play any strategy in his current support with probability one, all other agents remain indifferent between the strategies in their support.}. This is a strong refinement of the Nash equilibrium property and in two agent coordination games under our genericity assumption it coincides with the notion of pure Nash.
Since
 mixed equilibria are linearly unstable  by applying the center stable manifold theorem we establish that locally the set of initial conditions that converge to such an equilibrium is of measure zero. To translate this to a global statement about the size of the region of attraction technical smoothness conditions must be established about the discrete time map. For continuous time systems, such as the replicator \cite{Kleinberg09multiplicativeupdates}, these are standard. \cite{Kleinberg09multiplicativeupdates} also studies other discrete versions of replicator, in fact MWUA in its standard form, but there noise is injected into the system to simplify the stability analysis. Our analysis  does not require any additive noise. Also, our system is deterministic, implying a stronger convergence result.

 In the case of coordination games with isolated equilibria our theorem follows by combining the zero measure regions of attraction of all unstable equilibria via union bound arguments.
The case of uncountably infinite equilibria is tricky and requires specialized arguments. Intuitively the problem lies on the fact that
a) black box union bound arguments do not suffice, b) the standard convergence results in potential games merely imply convergence to
equilibrium sets, \textit{i.e.}, the distance of the state from the set of equilibria goes to zero, instead of the stronger point-wise
convergence, \textit{i.e.}, every trajectory has a unique (equilibrium) limit point. Set-wise convergence allows for complicated
non-local trajectories that weave infinitely often  in and out of the neighborhood of an equilibrium making topological arguments hard.
Once point-wise convergence has been established, the continuum of equilibria can be chopped down into countable many pieces via
Lindel\H{o}f's lemma and once again standard union bound arguments suffice. Nagylaki \etal pointwise convergence result
\cite{Nagylaki2} does not apply here, because their hyperbolicity assumption is not satisfied. Further, assuming $s\rightarrow 0$, they
analyze a continuous time dynamical system governed by a differential equation. Unlike Nagylaki the system we analyze is discrete MWUA,
and establish point-wise convergence to pure Nash equilibria almost always following the work of Losert and Akin \cite{akin}, even if
hyperbolicity is not satisfied (uncountably many equilibria).
%

We close our paper with some technical observations about the speed of divergence from the set of unstable
equilibria as well as discussing an average case analysis approach for predicting the probabilities that we converge to any
of the pure equilibria given a random initial condition.
We believe that these observations could stimulate future work in
the area.
\smallskip

\noindent\textbf{Biological Interpretation.}
Our work sheds new light on the role of natural selection in haploid genetics. We show that
 natural selection acts as an antagonistic process to the preservation of genetic diversity. The long term preservation of
 genetic diversity needs to be safeguarded by evolutionary mechanisms which are orthogonal to natural selection such as mutations and
 speciation.  This view, although may appear linguistically puzzling at first, is completely compatible with the mixability
 interpretation of
\cite{PNAS1:Livnat16122008,PNAS2:Chastain16062014}.  Mixability  implies that good ``mixer'' alleles (\textit{i.e.}, alleles that enjoy high fitness in the current genotypic
landscape) gain an evolutionary advantage over their competition. On the other hand, the preservation of mixed populations relies on
this evolutionary race between alleles having no clear long term winner with the average-over-time mixability  of two, or more, alleles
being roughly equal\footnote{In game theoretic terms, in order for two strategies to be played with positive probability by the same agent in the long run, it must be the case that the time-average expected utilities of these two strategies are roughly equal. The time average here is over the history of play so far.}. As with actual races, ties are rare and hence mixability leads to non-mixed populations in the long run.

According to recent PNAS commentary \cite{barton2014diverse}  some of the points in \cite{PNAS2:Chastain16062014}  raised questions
when compared against commonly held  beliefs  in mathematical biology.

\begin{quote}
``Chastain et al. (1) suggest that the representation of selection as (partially) maximizing entropy may help us understand how
selection maintains diversity. However, it is widely believed that selection on haploids (the relevant case here) cannot maintain a
stable polymorphic equilibrium. There seems to be no formal proof of this in the population genetic literature\dots'' \end{quote}

\noindent
Our argument above helps bridge this gap between belief and theory.



\section{Related work}

\medskip

\noindent\textbf{MWUA, Genetics, Ecology and Biology.} The earliest connection, to our knowledge, between MWUA and genetics lies in \cite{Kleinberg09multiplicativeupdates}, where such a connection is established between MWUA (in its usual exponential form)  and replicator dynamics \cite{Taylor1978145,Schuster1983533}, one of the most basic tools in mathematical ecology, genetics, and mathematical theory of selection and evolution.  Specifically, MWUA is up to first order approximation equivalent to replicator dynamics. Since the MWUA variant examined in \cite{PNAS2:Chastain16062014} is an approximation of its standard exponential form, these results follow a unified theme. MWUA in its classic form is up to first order approximation equivalent to models of evolution. The MWUA variant examined in  \cite{PNAS2:Chastain16062014} was introduced by Losert and Akin in \cite{akin} in a paper that also brings biology and game theory together. Specifically, they use game theoretic analysis to prove the first point-wise convergence to equilibria for a class of evolutionary dynamics resolving an open question at the time. We build on the techniques of this paper, while also exploiting the (in)stability analysis of mixed equilibria along the lines of \cite{Kleinberg09multiplicativeupdates}.
The connection between MWUA and replicator dynamics by \cite{Kleinberg09multiplicativeupdates}  also immediately implies connections between MWUA and mathematical ecology. This is because replicator dynamics is known to be equivalent (up to a diffeomorphism) to the classic prey/predator population models of Lotka-Voltera \cite{Hofbauer98}.
\medskip

\noindent\textbf{MWUA (and variants) in game theory.}
As a result of the discrete nature of MWUA, its game theoretic analysis tends to be trickier than that of its
continuous time variant, the replicator. Analyzed settings of this family of dynamics include zero-sum games \cite{Akin84,Sato02042002}, potential (congestion) games \cite{Kleinberg09multiplicativeupdates}, games with non-converging behavior
 \cite{paperics11,NonConvergingDask,RePEc438,Balcan12,sigecom11} and as well as families of network games \cite{Soda14,PiliourasAAMAS2014}.
New techniques can predict analytically the limit point of replicator systems starting from randomly chosen initial condition. This approach is referred to as
 average case analysis of game dynamics \cite{Soda15a}.
\medskip

\noindent\textbf{Genetics and Computer Science.} In the last couple of years we have witnessed an accumulation of papers and   problem
proposals in the intersection of computer science and genetics
\cite{evolfocs14,Vishnoi:2013:MER:2422436.2422445,VishnoiOpen,VishnoiSpeed,DSV}. In the
closing sections of our paper, we add to this exciting discussion by pointing out some new challenges along these lines.

\section{Preliminaries}
In this section we formally describe the two player coordination games, the dynamics under consideration, and its equivalence with
MWUA in evolution. First we start with some notations.
\medskip

\noindent{\bf Notations:}  All vectors are in bold-face letters, and are considered as column vectors. To denote a row
vector we use $\xx^T$. The $i^{th}$ coordinate of $\xx$ is denoted by $x_i$, and for $l<k$, $\xx(l:k)$ denote subvector
$(x_l,x_{l+1},\dots,x_k)$. For two vectors $\xx,\yy$ let $(\xx;\yy)$ denote the concatenation of two vectors. For $k\in \Real$,
$\kk_{n\times n}$ represents $n\times n$ matrix with all entries set to $k$. We denote set $\{1,\dots,n\}$ by $[1:n]$.
$int \; A$ is the interior of set $A$.


\subsection{Two-player Games and Nash equilibrium}
In this paper we consider two-player games, where each player has finitely many pure strategies (moves).
Let $S_i,\ i=1,2$ be the set of strategies for player $i$, and let $m\defeq|S_1|$ and $n\defeq|S_2|$.  Then such a game can be
represented by two payoff matrices $A$ and $B$ of dimension $m\times n$, where payoff to the players are $A_{ij}$ and $B_{ij}$
respectively if the first-player plays $i$ and the second plays $j$.

Players may randomize among their strategies. 
The set of mixed strategies for the first player is $\Delta_1=\{\xx=(x_1,\dots,x_m)\ |\ \xx\ge 0, \sum_{i=1}^m x_i=1\}$, and for the
second player is $\Delta_2=\{\yy=(y_1,\dots, y_n)\ |\ \yy\ge 0, \sum_{j=1}^n y_j=1\}$.
The expected payoffs of the first-player and second-player from a mixed-strategy $(\xx,\yy)\in \D_1\times \D_2$ are, respectively
\[ \sum_{i,j} A_{ij} x_i y_j=\xx^TA\yy\ \ \ \ \mbox{ and }\ \ \ \ \sum_{i,j}B_{ij}x_iy_j = \xx^TB\yy\]

\begin{definition}{(Nash Equilibrium \cite{agt.bimatrix})}
A strategy profile is said to be a Nash equilibrium strategy profile (NESP) if no player achieves a better payoff by a
unilateral deviation \cite{nash}. Formally, $(\xx,\yy) \in \D_m\times \D_n$ is a NESP iff $\forall \xx' \in \D_m,\
\xx^TA\yy \geq \xx'^TA\yy$ and $\forall \yy' \in \D_n,\  \xx^TB\yy \geq \xx^TB\yy'$.
\end{definition}

Given strategy $\yy$ for the second-player, the first-player gets $(A\yy)_k$ from her $k^{th}$ strategy. Clearly, her best
strategies are $\argmax_k (A\yy)_k$, and a mixed strategy fetches the maximum payoff only if she randomizes among her best
strategies. Similarly, given $\xx$ for the first-player, the second-player gets $(\xx^TB)_k$ from $k^{th}$ strategy, and same
conclusion applies. These can be equivalently stated as the following complementarity type conditions,

\begin{equation}\label{eq.ne}
\begin{array}{ll}
\forall i \in S_1,\hspace{.06in} x_i>0 \ \ \Rightarrow\ \ & (A\yy)_i =\max_{k \in S_1} (A\yy)_k\\
\forall j \in S_2,\hspace{.06in} y_j>0 \ \ \Rightarrow & (\xx^TB)_j = \max_{k \in S_2} (\xx^TB)_k
\end{array}
\end{equation}

\noindent{\bf Symmetric Game.} Game $(A,B)$ is said to be symmetric if $B=A^T$. In a symmetric game the strategy sets of both the
players are identical, i.e., $m=n$, and $S_1=S_2$. We will use $n$, $S$ and $\D_n$ to denote number of strategies, the
strategy set and the mixed strategy set respectively of the players in such a game. A Nash equilibrium profile $(\xx,\yy)\in \D_n\times
\D_n$ is called {\em symmetric} if $\xx=\yy$.  Note that at a symmetric strategy profile $(\xx,\xx)$ both the players get payoff
$\xx^TA\xx$. Using (\ref{eq.ne}) it follows that $\xx \in X$ is a symmetric NE of game $(A,A^T)$, with payoff $\pi$ to both players, if
and only if,

\begin{equation}\label{eq.sne}
\forall i \in S, x_i>0 \Rightarrow (A\xx)_i = \max_k (A\xx)_k
\end{equation}

\noindent{\bf Coordination Game.} In a coordination game $B=A$, i.e., both the players get the same payoff regardless of who is playing
what. Note that such a game always has a pure equilibrium, namely $\argmax_{(i,j)} A_{ij}$.

\subsection{Discrete Dynamical Systems}
A dynamical system is a smooth action of the reals or the integers on another object (usually a manifold). When the integers are
acting, the system is called discrete and is given by the following update rule:$$\textbf{x}
(n+1) = f(\textbf{x}(n))$$ with $n \in
\mathbb{N} \textrm{ or } \mathbb{Z}$ where $f$ is called the rule/map of the dynamic. A point $x_{*}$ is called \textit{fixed point} or
\textit{equilibrium} if $f(\textbf{x}_{*}) = \textbf{x}_{*}$. A trajectory of the dynamical system is a (infinite) sequence
of vectors $\textbf{x}(0),f(\textbf{x}(0)),f^2(\textbf{x}(0)),...$ where $f^n$ is the composition of $f$ for $n$ times.
Dynamical system theory is the branch of mathematics that tries to understand the behavior of dynamical systems. To
understand their behavior, there are plenty of questions one needs to answer. Does the system converge? What is the rate of
convergence? Which are the stable fixed points?
\medskip
\medskip


\noindent{\bf Games and discrete MWUA of \cite{PNAS2:Chastain16062014}.}
Chastain et. al. \cite{PNAS2:Chastain16062014} observed that the update rule derived by Nagylaki \cite{Nagylaki1} for allele
frequencies, during evolutionary process under weak-selection, is exactly multiplicative weight update
algorithm (MWUA) applied on coordination game, where genes are players and alleles are their strategies.
Formally, if fitness values of a genome defined by a combination of alleles (strategy profile) is from
$[1-s,\ 1+s]$ for a small $s>0$ (weak selection), then for the two-gene (two-player) case such a fitness matrix can be written as
$B=\ones_{m\times n} + \eps C$, where each $C_{ij} \in \mathbb Z$ and $\eps<<1$. This, defines a coordination game $(B,B)$.
Further, the change in allele frequencies in each new generation is as per the following rule:
\begin{equation}\label{eq.evol}
\forall i\in S_1,\ x_i(t+1) = \frac{x_i(t) (1+\eps (C\yy(t))_i)}{1+\eps \xx(t)^TC\yy(t)};\ \ \ \forall j \in S_2,\ y_j(t+1) = \frac{y_j(t) (1+\eps (C^T\xx(t))_j)}{1+\eps
\xx(t)^TC\yy(t)}
\end{equation}
Using the fact that $B=\ones_{m\times n} + \eps C$, this can be reformulated as,

\[
\frac{x_i(t) (1+\eps (C\yy(t))_i)}{1+\eps \xx(t)^TC\yy(t)} = x_i(t) \frac{(B\yy(t))_i}{\xx^T(t)B\yy(t)};\ \ \ \frac{y_j(t) (1+\eps
(C^T\xx(t))_j)}{1+\eps \xx(t)^TC\yy(t)}=y_j(t) \frac{(B^T\xx(t))_j}{\xx^T(t)B\yy(t)} \]

Thus, in this paper we study convergence of discrete MWUA through this reformulation.This reformulation is also known as {\em discrete
replicator dynamics}. In general, given a game $(A,B)$ consider the update rule (map) $f:\D_m\times \D_n
\rightarrow \D_m\times \D_n$,

\begin{equation}\label{maindy}
\mbox{For $(\xx,\yy) \in \D_m\times \D_n$ \ \ \ if $(\xx',\yy')=f(\xx,\yy)$,\ \ \  then}\ \ \
\begin{array}{ll}
\forall i \in S_1,& x'_i = x_i\frac{(A\yy)_i}{\xx^TA\yy}\\
\forall j \in S_2,& y'_j = y_j\frac{(\xx^TB)_j}{\xx^TB\yy}
\end{array}
\end{equation}

Clearly, $\xx'\in \D_m,\yy'\in \D_n$, and therefore $f$ is well-defined. Starting with $(\xx(0),\yy(0))$, the strategy profile at time
$t\ge 1$ is $(\xx(t),\yy(t))=f(\xx(t-1),\yy(t-1))=f^t(\xx(0),\yy(0))$.
\medskip
\medskip

\noindent{\bf Losert and Akin \cite{akin}.} Losert and Akin showed a very interesting result on the convergence of $\dmwa$ 
when applied on evolutionary games \cite{HistoryofEvolutionIdea} with positive matrix. These games are symmetric games, where pure strategies are
species and the player is playing against itself, i.e., symmetric strategy ($\xx=\yy$).
Consider a symmetric game $(A,A^T)$ where $A$ is an $n \times n$ positive matrix, and the following dynamics
starting with $\zz(0) \in \Delta_n$.
\begin{equation}\label{onedynamic}
z_{i}(t+1) = z_{i}(t) \frac{(A\zz(t))_{i}}{\zz(t)^TA\zz(t)}
\end{equation}

Clearly, $\zz(t+1)\in \D_n,\ \forall t\ge 1$. Thus, there is a map $f_s:\D_n\rightarrow \D_n$ corresponding to the above dynamics, where
if $\zz'=f_s(\zz)$ then $z'_i = z_i \frac{(A\zz)_{i}}{\zz^TA\zz}$, implying
\begin{equation}\label{eq.fs}
\zz(t+1) = f_s(\zz(t))=f^{t+1}_s(\zz(0))
\end{equation}

If $\zz(t)$ is a fixed-point of $f_s$ then $\zz(t')=\zz(t),\ \forall t'\ge t$.
Losert and Akin \cite{akin} proved that the above dynamical system converges pointwise to
fixed-point, and that map $f$ is a diffeomorphism in an open set that contains $\Delta_{n}$. Formally:

\begin{thm}\label{losert} \cite{akin} Let $\{\zz(t)\}$ be an orbit for the dynamic of (\ref{onedynamic}). As $t$ approaches $\infty$,
$\zz(t)$ converges to a unique fixed-point $\textbf{q}$. 
Additionally, the map $f_s$ corresponding to (\ref{onedynamic}) is a
diffeomorphism, i.e. it is a one-to-one, onto, and smooth function whose inverse function is also smooth.  \end{thm}

\subsection{Terms used in biology}\label{sec:prelEvol}
We provide brief non-technical definitions of a few biological terms that we use in this paper.

\noindent{\bf Gene.}
A unit that determines some characteristic of the organism, and passes traits to offsprings.
All organisms have genes corresponding to various biological traits, some of which are instantly visible, such as eye color or number
of limbs, and some of which are not, such as blood type.
\medskip
\medskip

\noindent{\bf Allele.}
Allele is one of a number of alternative forms of the same gene, found at the same place on a chromosome,
Different alleles can result in different observable traits, such as different pigmentation.
\medskip
\medskip

\noindent{\bf Genotype.}
The genetic constitution of an individual organism.
\medskip
\medskip

\noindent{\bf Phenotype.}
The set of observable characteristics of an individual resulting from the interaction of its genotype with the environment.
\medskip
\medskip


\noindent{\bf Diploid.}
Diploid means having two copies of each chromosome. Almost all of the cells in the human body are diploid.
\medskip
\medskip

\noindent{\bf Haploid.}
A cell or nucleus having a single set of unpaired chromosomes.
Our sex cells (sperm and eggs) are haploid cells that are produced by meiosis. When sex cells unite during fertilization, the haploid
cells become a diploid cell.
\section{Pointwise Convergence and Diffeomorphism}\label{sec:conv}
In this section we show that the $\edmwa$ of (\ref{maindy}), 
when applied to a two-player coordination game $(B,B)$,  converges pointwise to a fixed-point of $f$ under weak selection.
Further, map $f$ is diffeomorphism. Essentially we will reduce the problem to applying $\dmwa$ 
on symmetric game with positive matrix and then use the result of Losert and Akin \cite{akin} (Theorem \ref{losert}).


Under weak selection regime we have $B_{ij} \in [1-s, 1+s], \ \forall (i,j)$, for some $s<1$. Let
 $\epsilon<1-s$, and consider the following matrix
\begin{equation}\label{eq.a}
A = \left[\begin{array}{cc} \eeps_{m\times m} & B-\eps \\ B^T-\eps & \eeps_{n \times n}\end{array}\right]
\end{equation}

%
%

We will show that applying dynamics of (\ref{maindy}) on game $(B, B^T)$ starting at $(\xx(0),\yy(0))$ is same as applying
(\ref{onedynamic}) on game $(A, A^T)$ starting at $\zz(0)=(\frac{\xx(0)}{2};\frac{\yy(0)}{2})$.

\begin{lem}\label{le:equivDy}
Given $(\xx(0),\yy(0))\in \D_1,\D_2$, let $\zz(0)=(\frac{\xx(0)}{2};\frac{\yy(0)}{2})$, then $\forall t \ge0$, $(\xx(t);\yy(t)) =
2*\zz(t)$, where $\xx(t)$ and $\yy(t)$ are as per (\ref{maindy}) and $\zz(t)$ is as per (\ref{onedynamic}).
\end{lem}
\begin{proof}
We will show the result by induction. By hypothesis the base case of $t=0$ holds. Suppose, it holds up to time $t$, then
let $\xx=\xx(t+1)$, $\yy=\yy(t+1)$ and $(\xx';\yy')=\zz(t+1)$. Now, $\forall i \le m+n,\ z_i(t+1) = z_i(t)
\frac{(A\zz(t))_i}{\zz(t)^TA\zz(t)}$ together with $\zz(t)=\frac{1}{2}(\xx(t);\yy(t))$ gives us
\[
\forall i \le m, x'_i = \frac{x_i(t)}{2}\frac{\frac{\eps\sum_i x_i(t)}{2} + \frac{(B\yy(t))_i}{2} - \frac{\eps\sum_j
y_j(t)}{2}}{\frac{\xx(t)^TB\yy(t)}{4} + \frac{\yy(t)B^T\xx(t)}{4}} = \frac{2x_i(t)}{4} \frac{(B\yy(t))_i}{\xx(t)^TB\yy(t)} =
\frac{x_i}{2}
\]

Similarly, we can show that $\forall j\le n,\ y'_j =\frac{y_j}{2}$, and the lemma follows.
\end{proof}

Lemmas 
\ref{le:equivDy} establishes equivalence between games $(B,B)$ and $(A,A^T)$ in terms of 
dynamics, and thus the next theorem follows using Theorem \ref{losert}.

\begin{thm}\label{convergence}
Let $\{\textbf{x}(t),\textbf{y}(t)\}$ be an orbit for the dynamic of (\ref{maindy}). As $t$ approaches $\infty$,
$(\textbf{x}(t),\textbf{y}(t))$ converges to a unique fixed-point $(\pp, \qq)$. 
Additionally, the map $F$ corresponding to (\ref{maindy}) is a diffeomorphism, i.e. it is a one-to-one, onto, smooth function whose
inverse function is also smooth.
\end{thm}


\section{Convergence to Pure NE Almost Always}\label{sec:pureNE}
In Section \ref{sec:conv} we saw that dynamics of (\ref{maindy}) converges to a fixed point regardless of where we start in coordination games
with weak-selection. However, which equilibrium it converges to depends on the starting point. In this section we show that it
almost always converge to a pure Nash equilibrium under mild genericity assumptions on the game matrix. In the light of
the known fact that a coordination game $(B,B)$, where $B_{ij}$s are chosen uniformly at random from $[1-s, 1+s]$, may have
exponentially many mixed NE \cite{Arxiv:DBLP:journals/corr/abs-1208-3160,ITCS:DBLP:dblp_conf/innovations/ChastainLPV13}, this result comes as a surprise.

To show the result, we use the concept of {\em weakly stable} Nash equilibrium  \cite{Kleinberg09multiplicativeupdates}. This is a refinement of the classic notion of equilibrium
and  we show that for coordination games it coincides with pure NE under some mild assumptions.
Further, we connect them to {\em stable} fixed-points of $f$ (\ref{maindy}) by showing that all stable fixed points of $f$ are weakly stable Nash equilibria.
Finally, using the {\em Center Stable Manifold Theorem}
\cite{shub} we show that dynamics defined by $f$ converges to {\em stable} fixed-points except for a zero-measure set of starting points.

\begin{mydef} \cite{Kleinberg09multiplicativeupdates}A Nash equilibrium $(\xx,\yy)$ is called {\em weakly stable} if fixing one of the players to choosing a pure strategy in the
support of her strategy with probability one, leaves the other player indifferent between the strategies in his support, e.g., let
$T_1$ and $T_2$ are supports of $\xx$ and $\yy$ respectively, then for any $i \in T_1$ if the first player plays $i$ with probability
one then the second player is indifferent between all the strategies of $T_2$, and vice-versa.
\end{mydef}

Note that {\em pure} NE are always weakly stable, and coordination games always have pure NE. Further, for a mixed-equilibrium to be
{\em weakly stable}, for any $i \in T_1$ all the $B{ij}$s corresponding to $j \in T_2$ are the same. Thus, the next lemma follows.

\begin{lem}\label{le:pureNE}
If coordinates of a row or a column of $B$ are all distinct, then every {\em weakly stable} equilibrium is a pure Nash equilibrium.
\end{lem}
\begin{proof}
To the contrary suppose $(\xx,\yy)$ is a mixed weakly stable NE, then for $T_1=\{i\ |\ x_i>0\}$ and $T_2=\{j\ |\ y_j>0\}$ we have
$\forall i\in T_1,\ B_{ij}=B{ij'},\ \forall j\neq j' \in T_2$, a contradiction.
\end{proof}

\begin{remark}
We note that the games analyzed in
\cite{Arxiv:DBLP:journals/corr/abs-1208-3160,ITCS:DBLP:dblp_conf/innovations/ChastainLPV13}, where entries of matrix $B$ are
chosen uniformly at random from the interval $[1-s, 1+s]$, will have distinct
entries in each of its rows/columns with probability one, and thereby due to Lemma
\ref{le:pureNE} all its {\em weakly stable} NE are pure NE.
\end{remark}

Stability of a fixed-point is defined based on eigenvalues of Jacobian matrix evaluated at the fixed-point. So let us first describe
the Jacobian matrix of function $f$. We denote this matrix by $J$ which is $m+n \times m+n$, and let $f_k$ denote the function that
outputs $k^{th}$ coordinate of $f$. Then, $\forall i \neq i'\le m$ and $\forall j \neq j' \le n$
\[
\begin{array}{ll}
J_{ii}= \frac{df_i}{dx_i} =  \frac{(By)_{i}}{x^TBy}-x_{i}\left(\frac{(By)_{i}}{x^TBy}\right)^2,
& J_{(m+j)(m+j)} = \frac{df_{m+j}}{dy_j} = \frac{(B^Tx)_{i}}{x^TBy}-y_{i}\left(\frac{(B^Tx)_{i}}{x^TBy}\right)^2\\
J_{ii'} = \frac{df_i}{dx_{i'}} = -x_{i}\frac{(By)_{i}\cdot (By)_{i'}}{(x^TBy)^2}, &
J_{(m+j)(m+j')} = \frac{df_{m+j}}{dy_{j'}} = -y_{j}\frac{(B^Tx)_{j}\cdot (B^Tx)_{j'}}{(x^TBy)^2}\\
J_{i(m+j)} = \frac{df_i}{dy_j} =  x_{i}\frac{B_{ij}\cdot (x^TBy) - (By)_{i}(B^Tx)_{j}}{(x^TBy)^2}, &
  J_{(m+j)i} = \frac{df_{m+j}}{dx_i} =  y_{j}\frac{B_{ij}\cdot (x^TBy) - (B^Tx)_{j}(By)_{i}}{(x^TBy)^2}
\end{array}
\]

Now in order to use {\em Center Stable Manifold Theorem} (see Theorem \ref{manifold}), we need a map whose domain is full-dimensional
around the fixed-point. However, an $n$-dimensional simplex ($\D_n$) in $\Real^n$ has dimension $n-1$, and therefore the domain of $f$,
namely $\D_m\times \D_n$ is of dimension $m+n-2$ in space $\Real^{m+n}$. Therefore, we need to take a projection of the domain space and
accordingly redefine the map $f$. We note that the projection we take will be fixed-point dependent; this is to keep of the proof of
Lemma \ref{weakly} relatively less involved later.

Let $\rr=(\pp,\qq)$ be a fixed-point of map $f$ in $\D_m\times \D_n$.  Define $i(\rr)$ and $j(\rr)$ to be coordinates of $\pp$ and
$\qq$ respectively that are non-zero, i.e. $p_{i(\rr)}>0$ and $q_{j(\rr)}>0$.  Consider the mapping $z_{\textbf{r}} : R^{m+n} \to
R^{m+n-2}$ so that we exclude from each player $1,2$ the variables $x_{i(\rr)}, y_{j(\rr)}$ respectively.  We substitute the variables
$x_{i(\rr)}$ with $1-\sum_{i\neq i(\rr)} x_i$ and $y_{j(\rr)}$ with $1-\sum_{j \neq j(\rr)} y_j$.  Consider map $f$ under the
projection $z_{\rr}$, and let $J^{\rr}$ denote the {\em projected Jacobian} at $\rr$. Then, $\forall i,i'\in [1:m]\setminus\{i(\rr)\}$
and $\forall j,j'\in [1:n]\setminus\{j(\rr)\}$,

\begin{equation}
\begin{array}{ll}
J^{\textbf{r}}_{ii} =&  \frac{(By)_{i}}{x^TBy}-x_{i}\left(\frac{(By)_{i}}{x^TBy}\right)^2+x_{i}\frac{(By)_{i}\cdot (By)_{i(\rr)}}{(x^TBy)^2}\\
J^{\textbf{r}}_{(m+j)(m+j)} =&  \frac{(B^Tx)_{j}}{x^TBy}-y_{j}\left(\frac{(B^Tx)_{j}}{x^TBy}\right)^2+y_{j}\frac{(B^Tx)_{j}\cdot
(B^Tx)_{j(\rr)}}{(x^TBy)^2}\\
J^{\textbf{r}}_{ii'} = &-x_{i}\frac{(By)_{i}\cdot (By)_{i'}}{(x^TBy)^2}+x_{i}\frac{(By)_{i}\cdot (By)_{i(\rr)}}{(x^TBy)^2}\\
J^{\textbf{r}}_{(m+j)(m+j')} =& -y_{j}\frac{(B^Tx)_{j}\cdot (B^Tx)_{j'}}{(x^TBy)^2}+y_{j}\frac{(B^Tx)_{j}\cdot (B^Tx)_{j(\rr)}}{(x^TBy)^2}\\
J^{\textbf{r}}_{ij} = & x_{i}\frac{B_{ij}\cdot (x^TBy) - (By)_{i}(B^Tx)_{j}}{(x^TBy)^2}-x_{i}\frac{B_{ij(\rr)}\cdot (x^TBy) -
(By)_{i}(B^Tx)_{j(\rr)}}{(x^TBy)^2}\\
J^{\textbf{r}}_{(m+j)i} =&  y_{j}\frac{B_{ij}\cdot (x^TBy) - (B^Tx)_{j}(By)_{i}}{(x^TBy)^2}-y_{j}\frac{B_{i(\rr)j}\cdot (x^TBy) -
(B^Tx)_{j}(By)_{i(\rr)}}{(x^TBy)^2}
\end{array}
\end {equation}

The characteristic polynomial of $J^{\textbf{r}}$ at $\textbf{r}$ is  $$\prod_{i : p_{i}=0}\left(\lambda
-\frac{(B\textbf{q})_{i}}{\textbf{p}^TB\textbf{q}} \right)\prod_{i : q_{j}=0}\left(\lambda
-\frac{(B^T\textbf{p})_{j}}{\textbf{p}^TB\textbf{q}}\right)\times det(\lambda I -\mathbb{J}^{\textbf{r}})$$ where
$\mathbb{J}^{\textbf{r}}$ corresponds to $J^{\textbf{r}}$ at $\textbf{r}$ by deleting rows $i$ ,columns $j$ such that $p_i=0$ and $q_{j}=0$.\\

\begin{mydef} A fixed point $\textbf{r}$ is called linearly stable, if the eigenvalues of $J^{\textbf{r}}$ at $\textbf{r}$ have absolute value
less than or equal to 1. Otherwise it is called linear unstable.
\end{mydef}

The definition above is a slight modification of the classic definition of a stable fixed point, and has been tailored so that use of
Theorem \ref{manifold} becomes easier. The intuition here is that linearly unstable fixed points are going to be discarded by the dynamics in a robust manner, so it suffices to characterize the set of linearly stable fixed points.
Throughout the paper when we refer to (un)stable fixed points, we refer to this definition of stability.

\begin{lem}\label{le:sFP2NE}
Every linearly stable fixed point is a Nash Equilibrium.
\end{lem}
\begin{proof} Assume that a linearly stable fixed point $\textbf{r}$ is not a Nash equilibrium. Without loss of generality suppose player $t=1$
can deviate and gain. Since $\rr$ is a fixed-point of map $f$, $\forall p_i>0\ \Rightarrow (B\qq)_i = \pp^TB\qq$.
Hence, there exists a strategy $i \le m$ such that $p_i=0$ and $(B\qq)_i> \pp^TB\qq$. Then the characteristic polynomial has
$\frac{(B\qq)_i}{\pp^TB\qq} >1$ as a root, a contradiction.
\end{proof}

We are going to show that the dynamics of (\ref{maindy}) converge to linearly stable fixed-point except for measure zero starting conditions.
However, what we want is that it almost always converge to weakly stable NE. So, let us first establish relation between
stable fixed-points and weakly stable NE.

\begin{lem}\label{weakly}Every linearly stable fixed point is a weakly stable Nash equilibrium.
\end{lem}
\begin{proof}
Let $k\times k$ be the size of matrix $\RJ^{\rr}$. If $k=0$ then the equilibrium is pure and therefore is stable.
For the case when $k>0$, let $T_{\pp}$ and $T_{\qq}$ be the support of $\pp$ and $\qq$ respectively, i.e., $T_{\pp}=\{i\ |\ p_i>0\}$ and
similarly $T_{\qq}$. If we show that $\forall i,i'\in T_{\pp}$ and $\forall j,j'\in T_{\qq},\
M^{i,i',j,j'} = (B_{ij} - B_{i'j})-(B_{ij'} - B_{i'j'})=0$, then using argument similar to Theorem 3.8 in
\cite{Kleinberg09multiplicativeupdates}, the lemma follows. We show this using the expression of $tr((\RJ^{\rr})^2)$.

\begin{claim}
$tr((\mathbb{J}^{\textbf{r}})^2)  =$
\[
\begin{array}{l}
k + \frac{1}{(\textbf{p}^TB\textbf{q})^2} \displaystyle\sum_{ i<i ': i, i ' \neq i(\rr) \atop j < j ' : j, j
'\neq j(\rr)} {p_{i}q_{j}p_{i '}q_{j '} (M^{i,i',j,j'})^2 }
+  \frac{1}{(\textbf{p}^TB\textbf{q})^2} \sum_{ i: i \neq i(\rr) \atop j : j \neq j(\rr)} p_{i}q_{j}p_{i(\rr)}q_{j(\rr)} (M^{i,
i(\rr),j, j(\rr)})^2\\
 +   \frac{1}{(\textbf{p}^TB\textbf{q})^2}\displaystyle\sum_{ j<j':j,j' \neq j(\rr)  \atop i: i \neq i(\rr)}{p_{i}p_{i(\rr)}q_{j}q_{j'}(M^{i,
i(\rr),j, j '})^2} +   \frac{1}{(\textbf{p}^TB\textbf{q})^2}\sum_{ j:j \neq j(\rr) \atop i<i ': i, i ' \neq
i(\rr)}p_{i}p_{i'}q_{j}q_{j(\rr)}(M^{i, i ',j, j(\rr)})^2
\end{array}
\]
\end{claim}
\begin{proof}
Since $\mathbb{J}^{\textbf{r}} _{ii'}=0$, $\mathbb{J}^{\textbf{r}}_{(m+j)(m+j')}=0$ for $i \neq i'$ and $j\neq j'$,
and $\mathbb{J}^{\textbf{r}}_{ii}=1$, $\mathbb{J}^{\textbf{r}}_{(m+j)(m+j)}=1$
we get that $$tr((\mathbb{J}^{\textbf{r}})^2) = k+\sum_{i,j
}\mathbb{J}^{\textbf{r}}_{i(m+j)}\mathbb{J}^{\textbf{r}}_{(m+j)i}$$
We consider the following cases:
\begin{itemize}
\item Let $i<i '$ with $i,i' \neq i(\rr)$ and $j<j '$ with $j,j' \neq j(\rr)$ and we examine the term
$\frac{1}{(\textbf{p}^TB\textbf{q})^2}p_{i}q_{j}p_{i '}q_{j '}$ in the sum and we get that it appears with
\begin{align*}
&[[ M^{i,i',j, j(\rr)} ]  \times [ M^{i,i(\rr),j,j' } ]  +
[ M^{i,i ',j, j(\rr) } ]  \times [ M^{i(\rr),i ' ,j  ,j' } ]  \\+&
[M^{i,i' , j(\rr),j ' } ]  \times
[ M^{i,i(\rr),j,j ' }]  +
[M^{i,i ',j(rr),j ' } ]  \times
[ M^{ i(\rr), i',j , j ' }]  \\ =& (M^{i,i',j, j '})^2
\end{align*}

\item Let $i \neq i(\rr)$ and $j \neq j(\rr)$. The term $\frac{1}{(\textbf{p}^TB\textbf{q})^2}p_{i}q_{j}p_{i(\rr)}q_{j(\rr)}$ in the
sum appears in multiplication with $(M^{i,i(\rr),j, j(\rr)} )^2$.

\item Let $i<i '$ with $i,i' \neq i(\rr)$ and $j \neq j(\rr)$. The term $\frac{1}{(\textbf{p}^TB\textbf{q})^2}p_{i}q_{j}p_{i'}
q_{j(\rr)}$ in the sum appears with
\begin{align*}
&[M^{i, i', j ,j(\rr)}]  \times
[M^{i, i(\rr) ,j,j(\rr)}]   +
[M^{i, i' ,j, j(\rr)} ]  \times
[ M^{  i(\rr),i' ,j,j(\rr)} ]
\\& = (M^{i,i',j, j(\rr)})^2
\end{align*}

\item Similarly to the previous case, for
$j<j '$ with $j,j' \neq j(\rr)$ and $i \neq i(\rr)$. The term $\frac{1}{(\textbf{p}^TB\textbf{q})^2}p_{i}q_{j}p_{i(\rr)}
q_{j'}$ in the sum appears with  $(M^{i,i(\rr),j,j'})^2$.
\end{itemize}
\end{proof}

Trace of $(\mathbb{J}^{\textbf{r}})^2$ can not be larger than $k$, otherwise there exists an eigenvalue with absolute value greater
than one contradicting $\rr$ being a stable fixed-point. From the above claim, it is clear that $tr((\mathbb{J}^{\textbf{q}})^2)\ge k$
and it is exactly $k$ if and only if $M^{i,i',j,j'}=0$, $\forall i,i' \in T_1$ and $j,j' \in T_2$, and the lemma follows.
\end{proof}

In Appendix \ref{a:zero} we show that except for zero measure starting points $(\xx(0),\yy(0))$ the dynamics of (\ref{maindy}) converges to
stable fixed-points using the {\em Center Stable Manifold Theorem}, which proves the next theorem.

\begin{thm}\label{zero} The set of initial conditions in $\Delta_{m}\times \D_n$ so that the dynamical system converges to unstable
fixed points has measure zero.
\end{thm}

Theorem \ref{zero} together with Lemmas \ref{le:pureNE} and \ref{weakly} gives the following main result.

\begin{thm}\label{mainres} For all but measure zero initial conditions in $\Delta_{m}\times \D_n$, the dynamical system (\ref{maindy})
when applied to a coordination game $(B,B)$ with $B_{ij} \in [1-s, 1+s],\ \forall(i,j)$ for $s<1$, converges to weakly stable Nash
equilibria. Furthermore, assuming that entries in each row and column of $B$ are distinct, it converges to pure Nash equilibria.
\end{thm}

\section{Conclusion}

We show that standard  mathematical models of haploid evolution imply the
extinction of genetic diversity in the long term limit. This reflects a widely believed  conjecture in population genetics
\cite{barton2014diverse}.  We prove this via recent established connections between game theory, learning theory and genetics \cite{Arxiv:DBLP:journals/corr/abs-1208-3160,ITCS:DBLP:dblp_conf/innovations/ChastainLPV13,PNAS2:Chastain16062014}. Specifically,
in game theoretic terms we show that in the case of coordination games, under minimal genericity
assumptions, discrete
 MWUA converges to pure Nash equilibria for all but a zero measure of initial conditions. This result holds despite the fact
 that mixed Nash equilibria can be exponentially (or even uncountably) many, completely dominating in number the set of pure
 Nash equilibria.
Thus, in haploid organisms the long term preservation of
 genetic diversity needs to be safeguarded by other evolutionary mechanisms such as mutations and
  speciation.

The intersection between computer science, genetics and game theory has already provided some unexpected results and interesting novel connections. As these connections become clearer, new questions emerge alongside the possibility of transferring knowledge between these areas. In appendix \ref{sec:disc} we raise some novel questions that have to do with speed of dynamics as well as the possibility of understanding the evolution of biological systems given random initial conditions. Such an approach can be thought of as a middle ground between Price of Anarchy (worst case scenario) and Price of Stability (best case scenario) in game theory. We believe that this approach can also be useful from the standard game theoretic lens \cite{Soda15a}.

\section{Acknowledgments}

\noindent
We would like to thank Prasad Tetali for helpful discussions and suggestions.

\bibliographystyle{plain}
\bibliography{sigproc4}

\newpage
\appendix

\section{Proof of Theorem \ref{zero}}\label{a:zero}
To prove Theorem \ref{zero}, we will make use of the following important theorem in dynamical systems.

\begin{thm}\label{manifold}(Center and Stable Manifolds, p. 65 of \cite{shub})
Let $\textbf{p}$ be a fixed point for the $C^r$ local diffeomorphism $h: U \to \mathbb{R}^n$ where $U \subset \mathbb{R}^n$ is an open
(full-dimensional) neighborhood of $\pp$ in $\mathbb{R}^n$ and $r \geq 1$. Let $E^s \oplus E^c \oplus E^u$ be the invariant splitting
of $\mathbb{R}^n$ into generalized eigenspaces of $Dh(\textbf{p})$\footnote{Jacobian of $h$ evaluated at $\pp$} corresponding to
eigenvalues of absolute value less than one, equal to one, and greater than one. To the $Dh(\textbf{p})$ invariant subspace $E^s\oplus
E^c$ there is an associated local $h$ invariant $C^r$ embedded disc $W^{sc}_{loc}$ of dimension $dim(E^s \oplus E^c)$, and 
ball $B$ around $\pp$ such that:
\begin{equation} h(W^{sc}_{loc}) \cap B \subset W^{sc}_{loc}.\textrm{  If } h^n(\textbf{x}) \in B \textrm{ for all }n \geq 0,
\textrm{ then }\textbf{x} \in W^{sc}_{loc}
\end{equation}
\end{thm}

To use the theorem above we need to project the vector field to a lower dimensional space.  We consider the (diffeomorphism) function
$g$ that is a projection of the points $(\textbf{x},\textbf{y}) \in \mathbb{R}^{m+n}$ to $\mathbb{R}^{m+n-2}$ by excluding a specific
(the "first") variable for each player (we know that the probabilities must sum up to one for each player). Let $N=m+n$, then
we denote this projection of $\Delta_{N}$ by $g(\Delta_{N})$, i.e., $(\xx,\yy) \to_g (\xx',\yy')$ where $\xx'=(x_2,\dots,x_n)$ and
$\yy'=(y_2,\dots,y_n)$. Further, recall the fixed-point dependent projection function $z_{\rr}$ defined in Section \ref{sec:pureNE},
where we remove $x_{i(\rr)}$ and $y_{j(\rr)}$.
\\\\Map $f$ is one corresponding to dynamical system (\ref{maindy}). For an unstable fixed point $\textbf{r}$ we consider the function
$\psi_{\textbf{r}}(\vv) = z_{\textbf{r}}\circ f \circ z_{\textbf{r}}^{-1}(\vv)$ which is $C^1$ local diffeomorphism (due to theorem \ref{convergence} we know that the rule of the dynamical system is a diffeomorphism),
$(\vv) \in R^{N-2}$.
Let $B_{\rr}$ be the (open) ball that is derived from Theorem \ref{manifold} and we consider the union of these balls
(transformed in $\mathbb{R}^{N-2}$) $$A = \cup _{\textbf{r}}A_{\textbf{r}}$$ where $A_{\textbf{r}} =
g(z_{\textbf{r}}^{-1}(B_{\textbf{r}}))$ ($z^{-1}_{\textbf{r}}$ "returns" the set $B_{\textbf{r}}$ back to
$\mathbb{R}^{N}$). Set $A_{\textbf{r}}$ is an open subset of $\mathbb{R}^{N-2}$ (by continuity of $z_{\textbf{r}}$).  Taking
advantage of separability of $\mathbb{R}^{N-2}$ we have the following theorem.

\begin{thm} (Lindel\H{o}f's lemma) For every open cover there is a countable subcover.
\end{thm}

Therefore due to the above theorem, we can find a countable subcover for $A$, i.e.,
there exists fixed-points $\rr_1,\rr_2,\dots$ such that $A = \cup _{m=1}^{\infty}A_{\textbf{r}_{m}}$.
\\\\ For a $t \in \mathbb{N}$ let $\psi_{t,\textbf{r}}(\vv)$ the point after $t$ iteration of dynamics (\ref{maindy}), starting
with $\vv$, under projection $z_{\rr}$, i.e., $\psi_{t,\textbf{r}}(\vv) = z_{\textbf{r}}\circ f^{t} \circ
z_{\textbf{r}}^{-1}(\vv)$.
If point $(\vv) \in int \; g(\Delta_{N})$ (which corresponds to $g^{-1}(\vv)$ in our original $\Delta_{N}$) has as unstable fixed point
as a limit, there must exist a $t_{0}$ and $m$ so that $\psi_{t,\textbf{r}_{m}} \circ z_{\textbf{r}_{m}} \circ
g^{-1}(\vv) \in B_{\textbf{r}_{m}}$ for all $t \geq t_{0}$ (we have point-wise convergence from theorem \ref{convergence}) and therefore again from Theorem \ref{manifold} we get that
$\psi_{t_{0},\textbf{r}_{m}} \circ z_{\textbf{r}_{m}} \circ g^{-1}(\vv) \in W_{loc}^{sc}(\rr_m)$,
hence $\vv \in g\circ z^{-1}_{\textbf{r}_{m}} \circ
\psi^{-1}_{t_{0},\textbf{r}_{m}}(W_{loc}^{sc}(\rr_m))$.\\\\ Hence the set of points in $int \;
g(\Delta_{N})$ whose $\omega$-limit has an unstable equilibrium is a subset of
\begin{equation}
C=  \cup_{m=1}^{\infty} \cup_{t=1}^{\infty} g\circ z^{-1}_{\textbf{r}_{m}} \circ
\psi^{-1}_{t,\textbf{r}_{m}}(W_{loc}^{sc}(\rr_m))
\end{equation}

Since $\rr_m$ is unstable corresponding $dim(E^u)\ge 1$, and therefore dimension of $W_{loc}^{sc}(\rr_m)$ is at most $N-3$. Thus,
the manifold $W_{loc}^{sc}(\rr_m)$ has Lebesgue measure zero in $\mathbb{R}^{N-2}$. Finally since $g\circ z^{-1}_{\rr_{m}} \circ
\psi^{-1}_{t,\rr_{m}} : \mathbb{R}^{N-2} \to \mathbb{R}^{N-2}$ is continuously differentiable, $\psi_{t, \textbf{r}_{m}}$ is $C^1$ and
locally Lipschitz (see \cite{perko} p.71). Therefore using Lemma \ref{lips} below it preserves the null-sets, and thereby we get that
$C$ is a countable union of measure zero sets, i.e., is measure zero as well, and Theorem \ref{zero} follows.

\begin{lem}\label{lips} Let $g: \mathbb{R}^m \to \mathbb{R}^m$ be a locally  Lipschitz   function, then $g$ is null-set preserving, i.e., for $E \subset \mathbb{R}^m$ if $E$ has measure zero then $g(E)$ has also measure zero.
\end{lem}
\begin{proof}  Let $B_{\gamma}$ be an open ball such that $||g(\textbf{y}) - g(\textbf{x})|| \leq K_{\gamma} ||\textbf{y}-\textbf{x}||$
for all $\textbf{x},\textbf{y} \in B_{\gamma}$. We consider the union $\cup_{\gamma}B_{\gamma}$ which cover $\mathbb{R}^m$ by the
assumption that $g$ is locally Lipschitz. By Lindel\H{o}f's lemma we have a countable subcover, i.e., $\cup_{i=1}^{\infty}B_{i}$. Let
$E_{i} = E \cap B_{i}$. We will prove that $g(E_{i})$ has measure zero. Fix an $\epsilon >0$. Since $E_{i} \subset E$, we have that
$E_{i}$ has measure zero, hence we can find a countable cover of open balls $C_{1},C_{2},... $ for $E_{i}$, namely $E_{i} \subset
\cup_{j=1}^{\infty}C_{j}$ so that $C_{j} \subset B_{i}$ for all $j$ and also $\sum_{j=1}^{\infty} \mu(C_{j}) <
\frac{\epsilon}{K_{i}^m}$. Since $E_{i} \subset \cup_{j=1}^{\infty}C_{j}$ we get that $g(E_{i}) \subset \cup_{j=1}^{\infty}g(C_{j})$,
namely $g(C_{1}),g(C_{2}),...$ cover $g(E_{i})$ and also $g(C_{j}) \subset g(B_{i})$ for all $j$. Assuming that ball $C_{j} \equiv
B(\textbf{x},r)$ (center $\textbf{x}$ and radius $r$) then it is clear that $g(C_{j}) \subset B(g(\textbf{x}),K_{i} r)$ ($g$ maps the
center $\textbf{x}$ to $g(\textbf{x})$ and the radius $r$ to $K_{i}r$ because of Lipschitz assumption). But $\mu(B(g(\textbf{x}),K_{i}
r)) = K_{i}^m \mu(B(\textbf{x}, r)) = K_{i}^m \mu(C_{j})$, therefore $\mu(g(C_{j})) \leq K_{i}^m \mu(C_{j})$ and so we conclude that
$$\mu(g(E_{i})) \leq \sum_{j=1}^{\infty}\mu(g(C_{j})) \leq K_{i}^m \sum_{j=1}^{\infty}\mu(C_{j}) < \epsilon$$ Since $\epsilon$ was
arbitrary, it follows that $\mu(g(E_{i})) =0$. To finish the proof, observe that $g(E) = \cup_{i=1}^{\infty} g(E_{i})$ therefore
$\mu(g(E)) \leq \sum_{i=1}^{\infty} \mu(g(E_{i})) =0$.  \end{proof}

\section{Figure of stable/unstable manifolds in simple example}\label{sec:fig}

The figure \ref{fig:gamew3} corresponds to a two agent coordination game with payoff structure $B=\left[\begin{array}{cc}1 & 0 \\ 0 & 3
\end{array}\right]$. Since this game has two agents with two strategies each, in order to capture the state space of game it suffices
to describe one number for each agent, namely the probability with which he will play his first strategy. This game has three Nash
equilibria, two pure ones $(0,0), (1,1)$ and a mixed one $(\frac34,\frac34)$. We depict them using small circles in the figure.  The
mixed equilibrium has a stable manifold of zero measure that we depict with a black line.  In contrast, each pure Nash equilibrium has
region of attraction of positive measure.  The stable manifold of the mixed NE separates the regions of attraction of the two pure
equilibria.  The $(0,0)$ equilibrium has larger region of attraction, represented by darker region in the figure. It is the risk
dominant equilibrium of the game. Recently, in \cite{Soda15a} techniques have been developed to compute such objects (stable manifolds,
volumes of region of attraction) analytically.

\begin{figure}[!htb]
\centering
  \includegraphics[width=.43\linewidth]{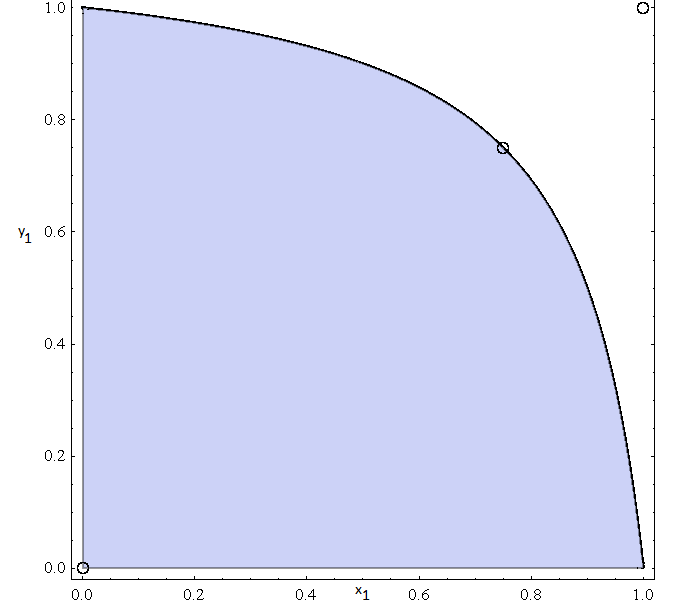}
  \caption{Regions of attraction for $B=[1\ 0;0\ 3]$, where $\circ$ correspond to NE points.} 
  \label{fig:gamew3}
\end{figure}

\section{Discussion}
\label{sec:disc}
Building on the observation of \cite{PNAS2:Chastain16062014} that the process of natural selection under weak-selection regime
can be modeled as discrete Multiplicative weight update dynamics on coordination games, we showed that it
converges to pure NE almost always in the case of two-player games.
As a consequence natural selection alone seem to lead to
extinction of genetic diversity in the long term limit, a widely believed conjecture of haploid genetics \cite{barton2014diverse}.
Thus, the long term preservation of genetic diversity must be safeguarded by evolutionary mechanisms which are orthogonal to natural
selection such as mutations and speciation.
This calls for modeling and study of these latter phenomenon in game theoretic terms under $\dmwa$.

Additionally below we observe that in some special cases, $(i)$ the rate of convergence of $\dmwa$ is doubly exponentially fast in some
special cases, and $(ii)$ the expected fitness of the resulting population, starting with a random distribution, under such dynamics is
constant factor away from the optimum fitness. It will be interesting to get similar results for the general case of two-player
coordination games.

\noindent{\bf Rate of Convergence.}
Let's consider a special case where $B$ is a square diagonal matrix.
In that case, starting from any point $(\textbf{x}(0),\textbf{y}(0))$ observe that after one time step, we get that
$\textbf{x}(1)=\textbf{y}(1)$ (i.e $f(\textbf{x}(0))=f(\textbf{y}(0))$). Therefore without loss of generality let us assume that
$\xx(0)=\yy(0)$. Then both the players get the same payoff from each of their pure strategies in the first play as $B=B^T$. And thus it
follows that $f^n(\textbf{x}(0))=f^n(\textbf{y}(0))$ for all $n \geq 1$. Let $U_i(t)$ be the payoff that both gets from their $i^{th}$
strategy at time $t$ (both will get the same payoff). Suppose for $i\neq j$ we have $U_i(0)=cU_j(0)$, then
\[
\frac{U_i(t)}{U_j(t)} = \left(\frac{B_{ii}x_i(t-1)}{B_{jj}x_j(t-1)}\right)^2
=\left(\frac{U_i(t-1)}{U_j(t-1)}\right)^2 = \left(\frac{U_i(0)}{U_j(0)}\right)^{2^{t}} = c^{2^t}
\]

Thus the ratio between payoffs from each pure strategy increases doubly exponentially, and the next lemma follows.

\begin{lem} If $z = \min_{j}
\frac{U_{i^*}(0)}{U_j(0)}$ where $i^* \in \argmax_k U_k(0)$, we get that after $O(\log \log \frac{1}{z\epsilon})$ we are
$\epsilon$-close to a Nash equilibrium with support $\argmax_k U_k(0)$ (in terms of the total variation distance).
\end{lem}

\subsection*{Average Price of Anarchy (APoA)}
Following the work of \cite{Soda15a} we can compute the {\em average price of anarchy} (APoA) 
for the following case, where $w>1$
$$B=\left[\begin{array}{cc}1 & 0 \\ 0 & w \end{array}\right]$$ 

Average Price of Anarchy (APoA) is defined w.r.t. a dynamics when the starting point is picked uniformly at random from $\D_m\times \D_n$. Dynamics under
different starting points may converge to different NE. Let {\em expected NE social welfare} be the expected social welfare (SW) at the
Nash equilibrium to which dynamics may converge, then $APoA=\frac{\mbox{Optimal SW}}{\mbox{Expected NE SW}}$. 

Since both the players have only two strategies, probability of the first strategy is enough to describe a profile. So let $(x,y)$
denote the probabilities with which both plays first strategy, i.e., $(x_1,y_1)$. Our game has three NE: $(1,1), (0,0)$ and
$(\frac{w}{1+1},\frac{w}{1+w})$. Since the set of starting point converging to the mixed NE has measure zero (Theorem \ref{mainres}),
we can ignore it. 
If $(x(0),y(0))=(x,y)$ is picked at random from $[0, 1]^2$ then let $A$ denote the area starting from where the dynamics (\ref{maindy})
converges to $(0,0)$ where the SW is $2w$. Then, $APoA = \frac{2w}{(2w*A) + 2(1-A)} = \frac{w}{wA+1-A}$. Next we compute $A$. 

As discussed above after first step strategies of both the players are same, and there after if $U_2(1)>U_1(1)$ then dynamics will
converge to $(0,0)$. 
\[
U_2(1)> U_1(1) \Leftrightarrow w^2(1-x)(1-y) > xy \Leftrightarrow y< \frac{w^2(1-x)}{x(1-w^2)+w^2}\] 
Thus $A$ is the area under the curve $y=\frac{w^2(1-x)}{w^2(1-x)+x}$, which is 
\begin{align*}
A = \int_{0}^1\frac{w^2(1-x)}{w^2(1-x)+x}dx &= [\frac{w^2x}{w^2-1}+(\frac{w^4}{1-w^2}+w^2)\frac{1}{1-w^2}\ln (w^2+x(1-w^2))]_{0}^1\\& =
\frac{w^4-w^2-w^2 \ln w^2}{(w^2-1)^2} 
\end{align*}   

Replacing $A$ in $APoA=\frac{w}{(w*A) + (1-A)}$, and setting its differentiation w.r.t. $w$ to zero gives,
\[
w^2 \ln{w^2}(2w^3 + w^2+1) = (w^2-1)(4w^3 + 3w^2 -2w-1)
\]

Solving the above gives the value of $w$ where APoA is maximum, and it turns out to be around $w=2.02$. APoA for $w=2.02$ is $1.1647$,
and thus the next lemma follows.

\begin{lem}
For the class of coordination games $(B,B)$ where $B=\left[\begin{array}{cc} 1 & 0 \\0 & w\end{array}\right]$ and $w>1$, the
APoA is at most $1.2$ under $\dmwa$.
\end{lem}

%
%

\end{document}